\documentclass[twoside]{amsart}

\usepackage{amssymb}
\usepackage{amsfonts}
\usepackage{amsmath}
\usepackage{amsthm}
\usepackage{xcolor}
\usepackage{graphicx}

\newtheorem{theorem}{Theorem}[section]

\newtheorem{mtheorem}[theorem]{Main Theorem}

\theoremstyle{definition}
\newtheorem{definition}[theorem]{Definition}

\theoremstyle{remark}
\newtheorem{remark}[theorem]{Remark}

\numberwithin{equation}{section}
\title[Counting parts divisible by $k$ in partitions]{Counting the parts divisible by $k$ in all the partitions of $n$ whose parts have multiplicity less than $k$}
\author[Herden, Sepanski, Stanfill, $\ldots$]{Daniel Herden, Mark R. Sepanski, Jonathan Stanfill, Cordell C. Hammon, Joel Henningsen, Henry Ickes, Jorge Marchena Menendez, Taylor Poe, Indalecio Ruiz, Edward L. Smith}
\email{daniel\_herden@baylor.edu, mark\_sepanski@baylor.edu}
\begin{document}



\begin{abstract}
Recent results by Andrews and Merca on the number of even parts in all partitions of $n$ into distinct parts, $a(n)$, were derived via generating functions. This paper extends these results to the number of parts divisible by $k$ in all the partitions of $n$ for which
the multiplicity of each part is strictly less than $k$, $a_k(n)$. Moreover, a combinatorial proof is provided using an extension of Glaisher’s bijection. Finally, we give the generating functions for this new family of integer sequences and use it to verify generalized pentagonal, triangular, and square power recurrence relations.
\end{abstract}

\maketitle
\tableofcontents


\section{Introduction}

Partitions of integers is a subject deeply intertwined with myriad subfields of mathematics and has been studied by a host of great mathematicians, Leibniz, Euler, Sylvester, Rogers, Hardy, MacMahon, Ramanujan, and Rademacher, to name only a few \cite{partitions}.

The celebrated Euler Partition Identity \cite[Corollary 1.2]{partitions} says that the number of partitions of $n$ into odd parts is equal to the number of partitions of $n$ into distinct parts. Glaisher \cite{G} found a beautiful combinatorial proof and extension of this result. Namely, for any $k \ge 2$, the number of partitions of $n$ with no part divisible by $k$ is equal to the number of partitions
of $n$ with each part repeated less than $k$ times. Euler's Identity is a special case of Glaisher's result for $k=2$. Franklin \cite{F} also found an extension for which Euler's Identity is a special case when $j = 0$.

\begin{theorem}[\cite{P}] \label{Franklin}
For any partition $\lambda = (\lambda_1^{m_1}, \lambda_2^{m_2}, \ldots, \lambda_\ell^{m_\ell})$,
denote by $\gamma_{\mathcal{O}}(\lambda)$ the number of even bases $\lambda_i$,
and by $\gamma_{\mathcal{D}}(\lambda)$ the number of repeated bases $\lambda_i$. Then, for $j \ge 0$,
the number of partitions of $n$ with $\gamma_{\mathcal{O}}(\lambda) = j$ is equal to the number
of partitions of $n$ with $\gamma_{\mathcal{D}}(\lambda) = j$.
\end{theorem}

There have been many related results that count partitions with particular properties.
Most notable among them may be the Rogers-Ramanujan Identity \cite[Corollary 7.6]{partitions} that the number of partitions of $n$ in which the difference between any two parts is at least $2$ equals the number of partitions of $n$ into parts $\equiv 1$ or $4$ mod $5$.

In this paper, we want to concentrate on Franklin's result. Theorem \ref{Franklin} has been rediscovered by Wilf \cite{W} and, more recently, by Andrews~\cite{A} and Fu, Tang \cite{FT}. In \cite{A}, Andrews used generating functions to prove a conjecture of Beck that the number of partitions of $n$ with $\gamma_{\mathcal{O}}(\lambda)=1$ equals the difference between the number of parts in all odd partitions of $n$ and the number of parts in all distinct partitions of~$n$. Later, Ballantine and Bielak \cite{BB} gave a combinatorial proof of Beck's conjecture, while Yang \cite{Y} and Li, Wang \cite{LW} provided generalizations. Li and Wang also investigated analogues
of Beck's conjecture for compositions \cite{LW2}. In~\cite{AM}, Andrews and Merca introduced new identities which link the number of even parts in the partitions of $n$ into distinct parts, $a(n)$, to partitions of $n$ with $\gamma_{\mathcal{O}}(\lambda)=1$.

In this note, we give a combinatorial proof of a generalization of the results in~\cite{AM}. In that paper, Andrews and Merca introduced four additional sequences: the number of partitions of $n$ into an odd (even, respectively) number of parts in which the set of even parts has only one element and the number of partitions of $n$ in which exactly one part is repeated and this part is odd (even, respectively).

Here, see \S \ref{sec: init defs}, we let
$$a_k(n)$$
count the number
of parts that are divisible by $k$ in all partitions of $n$ for which the multiplicity of each part is strictly less than $k$.
We let
$$b_{k,0}(n) \text{  and  } b_{k,1}(n)$$
count the partitions $\lambda = (\lambda_1^{m_1}, \lambda_2^{m_2}, \ldots, \lambda_\ell^{m_\ell})$ of $n$ satisfying (1) exactly one $\lambda_i$ is divisible by $k$ and (2) the corresponding exponent $m_i$ is divisible by $k$ ($b_{k,0}(n)$) or not divisible by $k$ ($b_{k,1}(n)$).
Finally, we let
$$ c_{k,0}(n) \text{  and  } c_{k,1}(n)$$
count the partitions $\lambda = (\lambda_1^{m_1}, \lambda_2^{m_2}, \ldots, \lambda_\ell^{m_\ell})$ of $n$ satisfying (1) exactly one $\lambda_i^{m_i}$ has $m_i \geq k$ and (2) the corresponding $\lambda_i$ is divisible by $k$ ($c_{k,0}(n)$) or not divisible by $k$ ($c_{k,1}(n)$).

Our main result (Theorem \ref{main}), generalizing \cite{AM} for $k=2$, is
	$$b_{k,0}(n)=c_{k,0}(n),$$
	$$b_{k,1}(n)=c_{k,1}(n),$$
	$$a_k(n) = b_{k,1}(n)-(k-1)b_{k,0}(n) = c_{k,1}(n)-(k-1)c_{k,0}(n).$$
Combinatorial proofs are given in \S \ref{sec:comb ids}.
We recently discovered that Li, Wang, \cite[Theorems 1.11 and 1.13]{LWpre}, already have a preprint that gives a very different proof of this result.

The sequence $a(n)$ from Andrews and Merca \cite{AM} appears in the OEIS as A116680. However, the family of integer sequences $a_k(n)$ appears to be new (for $k>2$) and quite interesting. A table of small values may be found in \S \ref{sec:gen fun}. In that same section, the generating function is also provided,
$$ a_k(n) = \prod_{n=1}^{\infty} \frac{1-q^{kn}}{1-q^n}\cdot
\sum_{n=1}^{\infty}
\frac{q^{kn} + 2q^{2kn} +\dots+ (k-1)q^{(k-1)kn}}{1 + q^{kn} + q^{2kn} +\dots+ q^{(k-1)kn}}.$$

In \S \ref{sec: recur reln}, using the generating function, we show that the sequences $a_k(n)$ satisfy the generalized pentagonal number recurrence relation when $k\nmid n$ (Theorem \ref{t5.2}),
\begin{align*}
	a_k(n)&=a_k(n-1)+a_k(n-2)-a_k(n-5)-a_k(n-7)\\
	&\quad\, +a_k(n-12)+a_k(n-15)-a_k(n-22)-a_k(n-26)+\cdots.
\end{align*}
They also satisfy the triangular number recurrence relation when $k$ is even and $n$ is odd (Theorem \ref{t5.3}),
\begin{align*}
	a_k(n)&=a_k(n-1)+a_k(n-3)-a_k(n-6)-a_k(n-10)\\
	&\quad\, +a_k(n-15)+a_k(n-21)-a_k(n-28)-a_k(n-36)+\cdots,
\end{align*}
and, still for $k$ even and $n$ odd, the square power recurrence relation (Theorem \ref{t6.3}),
\begin{align*}
	a_k(n)&=a_k(n-1)+a_k(n-2)-a_k(n-4)-a_k(n-8)\\
	&\quad\, +a_k(n-9)+a_k(n-18)-a_k(n-16)-a_k(n-32)+\cdots,
\end{align*}
while for $k$ and $n$ both even we have the square power identity (Theorem \ref{t6.4}),
\begin{align*}
	& a_k(n-1)-a_k(n-4)+a_k(n-9)-a_k(n-16)+\cdots\\
	&\qquad\, = a_k(n-2)-a_k(n-8)+a_k(n-18)-a_k(n-32)+\cdots.
\end{align*}
We also show that this identity is shared with the partition function (Theorem \ref{t6.5}),
\begin{align*}
	& p(n-1)-p(n-4)+p(n-9)-p(n-16)+\cdots\\
	&\qquad\, = p(n-2)-p(n-8)+p(n-18)-p(n-32)+\cdots = [p(n) - p_{\mathcal{DO}}(n)]/2,
\end{align*}
for $n$ even, where $p_{\mathcal{DO}}(n)$ is the number of partitions of $n$ into distinct odd parts.

In addition, we give some thoughts and conjectures on some intriguing new possible recurrence relations.

\section{Initial Definitions and Statement of Main Theorem} \label{sec: init defs}

For $n\in\mathbb{Z}^{\ge 0}$, let $\mathcal{P}(n)$ be the set of partitions of $n$.
In general, for $\lambda \in \mathcal{P}(n)$, we write $$\lambda = (\lambda_1^{m_1}, \lambda_2^{m_2}, \ldots, \lambda_\ell^{m_\ell})$$
with distinct $\lambda_i$, each $m_i \geq 1$, and $n = \sum_{i=1}^\ell m_i \lambda_i$, and we consider partitions equivalent up to reordering. We refer to $m_i$ as the \emph{multiplicity} of the part $\lambda_i$. When concatenating such sequences, we combine terms with the same base and add the corresponding exponents.

\begin{definition}
	Let $n,k\in \mathbb{Z}^{\ge 0}$ with $k\ge 2$. Let $$a_k(n)$$ count the number
	of parts (with multiplicity) that are divisible by $k$ in all partitions of $n$ for which the multiplicity of each part is strictly less than $k$.
\end{definition}

For example, the partitions of $n = 9$ are
\begin{align*}
	& (9),\, (8,1),\, (7,2),\, (7,1^2),\, (6,3),\,
	(6,2,1),\, (6,1^3),\, (5,4),\, (5,3,1),\, (5,2^2),\, \\
	&
	(5,2,1^2),\, (5,1^4),\, (4^2,1),\, (4,3,2),\,
	(4,3,1^2),\, (4,2^2,1),\, (4,2,1^3),\, (4,1^5),\,  \\
	&
	(3^3),\, (3^2,2,1),\, (3^2,1^3),\, (3,2^3),\,
	(3,2^2,1^2),\, (3,2,1^4),\, (3,1^6),\, (2^4,1),\,    \\
	&
	(2^3,1^3),\, (2^2,1^5),\, (2,1^7),\, (1^9).
\end{align*}
The partitions of $9$ with the multiplicity of each part less than $k = 3$ are
\begin{align*}
	& (9),\, (8,1),\, (7,2),\, (7,1^2),\, (6,3),\,
	(6,2,1),\,  (5,4),\, (5,3,1),\, (5,2^2),\, \\
	&
	(5,2,1^2),\, (4^2,1),\, (4,3,2),\,
	(4,3,1^2),\, (4,2^2,1),\, (3^2,2,1),\, (3,2^2,1^2).
\end{align*}
Of these, the ones that have parts divisible by $3$ are
\begin{align*}
	& (9),\, (6,3),\, (6,2,1),\, (5,3,1),\, (4,3,2),\,  (4,3,1^2),\, (3^2,2,1),\, (3,2^2,1^2).
\end{align*}
As $a_3(9)$ counts the parts divisible by $3$ with multiplicity, we see that
$$ a_3(9) = 10. $$

\begin{definition}
	Let $n,k\in \mathbb{Z}^{\ge 0}$ with $k\ge 2$.
	\begin{itemize}
		\item[$(a)$] Let $$b_{k,0}(n)$$ count all $\lambda \in \mathcal{P}(n)$,
		$\lambda = (\lambda_1^{m_1}, \lambda_2^{m_2}, \ldots, \lambda_\ell^{m_\ell})$, satisfying
		\begin{itemize}
			\item[$(1)$] exactly one $\lambda_i$ is divisible by $k$, and
			\item[$(2)$] the corresponding exponent $m_i$ is divisible by $k$.
		\end{itemize}
		\item[$(b)$] Let $$b_{k,1}(n)$$ count all $\lambda \in \mathcal{P}(n)$,
		$\lambda = (\lambda_1^{m_1}, \lambda_2^{m_2}, \ldots, \lambda_\ell^{m_\ell})$, satisfying
		\begin{itemize}
			\item[$(1)$] exactly one $\lambda_i$ is divisible by $k$, and
			\item[$(2)$] the corresponding exponent $m_i$ is not divisible by $k$.
		\end{itemize}
	\end{itemize}
\end{definition}

For example, the partitions of $n=9$ with exactly one base $\lambda_i$ divisible by $k = 3$ are
\begin{align*}
	& (9),\,
	(6,2,1),\, (6,1^3),\,  (5,3,1),\,
	(4,3,2),\,
	(4,3,1^2),\,
	(3^3),\,   \\
	&
	(3^2,2,1),\, (3^2,1^3),\, (3,2^3),\,
	(3,2^2,1^2),\, (3,2,1^4),\, (3,1^6).
\end{align*}
Of these, the only one whose corresponding exponent $m_i$ is divisible by $3$ is
$$ (3^3) $$
so
$$ b_{3,0}(9) = 1 \qquad \mbox{and} \qquad b_{3,1}(9) = 12. $$

\begin{definition}
	Let $n,k\in \mathbb{Z}^{\ge 0}$ with $k\ge 2$.
	\begin{itemize}
		\item[$(a)$] Let $$c_{k,0}(n)$$ count all $\lambda \in \mathcal{P}(n)$,
		$\lambda = (\lambda_1^{m_1}, \lambda_2^{m_2}, \ldots, \lambda_\ell^{m_\ell})$, satisfying
		\begin{itemize}
			\item[$(1)$] exactly one $\lambda_i^{m_i}$ has $m_i \geq k$, and
			\item[$(2)$] the corresponding $\lambda_i$ is divisible by $k$.
		\end{itemize}
		\item[$(b)$] Let $$c_{k,1}(n)$$ count all $\lambda \in \mathcal{P}(n)$,
		$\lambda = (\lambda_1^{m_1}, \lambda_2^{m_2}, \ldots, \lambda_\ell^{m_\ell})$, satisfying
		\begin{itemize}
			\item[$(1)$] exactly one $\lambda_i^{m_i}$ has $m_i \geq k$, and
			\item[$(2)$] the corresponding $\lambda_i$ is not divisible by $k$.
		\end{itemize}
	\end{itemize}
\end{definition}

For example, the partitions of $n=9$ with exactly one $\lambda_i$ with multiplicity $m_i$ at least $k = 3$ are
\begin{align*}
	& (6,1^3),\, (5,1^4),\, (4,2,1^3),\, (4,1^5),\,
	(3^3),\, (3^2,1^3),\, (3,2^3),\,   \\
	&
	(3,2,1^4),\, (3,1^6),\, (2^4,1),\,
	(2^2,1^5),\, (2,1^7),\, (1^9).
\end{align*}
Of these, the only one whose corresponding $\lambda_i$ is divisible by $3$ is
$$ (3^3) $$
so
$$ c_{3,0}(9) = 1 \qquad \mbox{and} \qquad c_{3,1}(9) = 12.$$

We will prove the following result in this paper. Li, Wang have a preprint, \cite[Theorems 1.11 and 1.13]{LWpre}, that gives an alternate proof of this result.

\begin{mtheorem} \label{main}
For all $n,k\in \mathbb{Z}^{\ge 0}$ with $k\ge 2$, the following holds:
	\begin{itemize}
		\item[$(a)$] $b_{k,0}(n)=c_{k,0}(n)$ and $b_{k,1}(n)=c_{k,1}(n)$
		\item[$(b)$] $a_k(n) = b_{k,1}(n)-(k-1)b_{k,0}(n) = c_{k,1}(n)-(k-1)c_{k,0}(n).$
	\end{itemize}
\end{mtheorem}

For $n=9$ and $k=3$, we have $a_3(9) = 10= b_{3,1}(9)-2b_{3,0}(9) = c_{3,1}(9)-2c_{3,0}(9).$

\section{Combinatorial Identities} \label{sec:comb ids}

\subsection{Combinatorial Definitions}

\begin{definition}	
		Let
		\begin{align*}
			A_k(n) &= \{
			(\lambda = (\lambda_1^{m_1}, \lambda_2^{m_2}, \ldots, \lambda_\ell^{m_\ell}), \lambda_{i_0}, m)
			\mid \\
			& \qquad\quad
			\lambda \in \mathcal{P}(n), \,
			m_i <k \, \forall i, \,
			1 \leq i_0 \leq \ell, \,
			k \mid \lambda_{i_0}, \,   
			0 \leq m < m_{i_0} \}.
		\end{align*}
	
		By construction, $$a_k(n) = |A_k(n)|.$$
		
		Let
		\begin{align*}
			A_k'(n) &= \{
			(\lambda = (\lambda_1^{m_1}, \lambda_2^{m_2}, \ldots, \lambda_\ell^{m_\ell}), \lambda_{i_0}, m)
			\mid \\
			& \lambda \in \mathcal{P}(n), \,
			1 \leq i_0 \leq \ell, \,
			m_i < k \, \forall i \not = i_0, \,
			m_{i_0} \ge k, \,
			k \mid \lambda_{i_0}, \,   
			0\le m < k \}.
		\end{align*}		
\end{definition}

For example, continuing with $n=9$ and $k=3$,
\begin{align*}
	A_3(9) &=\{
	((9),9,0),\, ((6,3),6,0),\, ((6,3),3,0), \,
	((6,2,1),6,0), \\
	&\quad\ ((5,3,1),3,0),\,
	((4,3,2),3,0),\,
	((4,3,1^2),3,0), \\
	&\quad\ ((3^2,2,1),3,0),\, ((3^2,2,1),3,1),\,
	((3,2^2,1^2),3,0)
	\}
\end{align*}
and
 \begin{align*}
 A'_3(9) =& \{
 ((3^3),3,0),((3^3),3,1),((3^3),3,2) \}.
 \end{align*}

\begin{definition}	
	Let
		\begin{align*}
			B_{k,0}(n) & = \{
			\lambda = (\lambda_1^{m_1}, \lambda_2^{m_2}, \ldots, \lambda_\ell^{m_\ell}, a^b) \in \mathcal{P}(n)
			\mid
			k \nmid \lambda_i \, \forall i, \,   
			k \mid a, \,   
			k \mid b       
			\},
		\end{align*}
		and
		\begin{align*}
			B_{k,1}(n) & = \{
			\lambda = (\lambda_1^{m_1}, \lambda_2^{m_2}, \ldots, \lambda_\ell^{m_\ell}, a^b) \in \mathcal{P}(n)
			\mid
			k \nmid \lambda_i \, \forall i, \,   
			k \mid a, \,   
			k \nmid b       
			\}.
		\end{align*}
		Note that $\ell$ may be zero above and that $b \geq 1$. By construction, we have for $m=0,1$, $$b_{k,m}(n) = |B_{k,m}(n)|.$$		
\end{definition}

Continuing with $n=9$ and $k=3$, we see that
$$ B_{3,0}(9) = \{(3^3)\} $$
and
\begin{align*}
B_{3,1}(9) &=\{ (9),\, (6,2,1),\, (6,1^3),\, (5,3,1),\, (4,3,2),\, (4,3,1^2),\, \\
	& \\&\quad\ (3^2,2,1),\, (3^2,1^3),\, (3,2^3),\, (3,2^2,1^2),\, (3,2,1^4),\, (3,1^6)\}.
\end{align*}

\begin{definition}
Let
		\begin{align*}
			C_{k,0}(n) & = \{
			\lambda = (\lambda_1^{m_1}, \lambda_2^{m_2}, \ldots, \lambda_\ell^{m_\ell}, a^b) \in \mathcal{P}(n)
			\mid
			m_i < k \, \forall i, 		 \,
			b \geq k, \,
			k \mid a
			\},
		\end{align*}
		and
		\begin{align*}
			C_{k,1}(n) & = \{
			\lambda = (\lambda_1^{m_1}, \lambda_2^{m_2}, \ldots, \lambda_\ell^{m_\ell}, a^b) \in \mathcal{P}(n)
			\mid
			m_i < k \, \forall i, 		 \,
			b \geq k, \,
			k \nmid a
			\}.
		\end{align*}
		Note that $\ell$ may be zero above and that $b \geq 1$. By construction, we have $$c_{k,m}(n) = |C_{k,m}(n)|$$ for $m=0,1$, and $$|A'_k(n)| = kc_{k,0}(n).$$
		
\end{definition}
Again with $n=9$ and $k=3$, we see that
\begin{align*}
C_{3,0}(9) =&\{(3^3)\}
\end{align*}
and
\begin{align*}
C_{3,1}(9) &= \{ (6,1^3),\, (5,1^4),\, (4,2,1^3),\, (4,1^5),\, (3^2,1^3), \, (3,2^3), \\&\quad\  (3,2,1^4),\, (3,1^6),\, (2^4,1), \, (2^2,1^5),\, (2,1^7),\, (1^9)\}.
\end{align*}

In the next sections, we will exhibit bijections between $$B_{k,m}(n) \text{  and  } C_{k,m}(n)$$
for $m=0,1$, and between $$A_k(n) \, \dot\cup \, A'_k(n) \text{  and  }  C_{k,0}(n) \, \dot\cup \, C_{k,1}(n).$$

\subsection{Bijection between $B_{k,m}(n)$ and $C_{k,m}(n)$ for $m=0,1$}

The following bijection is an extension of Glaisher's combinatorial proof \cite{G} of the Euler Partition Identity.

\begin{theorem} \label{B to C}
	Let $$\lambda = (\lambda_1^{m_1}, \lambda_2^{m_2}, \ldots, \lambda_\ell^{m_\ell}, a^b) \in B_{k,m}(n).$$ Expand each $m_i$ in base $k$ as
	$$ m_i = \sum_{j \in S_i} c_{ij} k^j $$
	with $0 < c_{ij} < k$ for $j$ in some index set $S_i$.
	
	Write $\frown$ for sequence concatenation. The mapping
	$$ \lambda \mapsto
	((k^j \lambda_i)^{c_{ij}})_{1 \leq i \leq \ell, \, j \in S_i} \frown (b^a) $$
	gives a bijection between
	$$ B_{k,m}(n) \longleftrightarrow C_{k,m}(n),$$
	for $m=0,1$. (Note that $b$ may also appear once among the $k^j \lambda_i$. In that case, combine those two terms with the same base by adding the corresponding exponents.)
\end{theorem}

\begin{proof}
	For injectivity, first note that $b$ in the statement of the Theorem is uniquely recovered from its image as the only part of the partition on the right hand side with multiplicity at least $k$. The value of $a$ is then uniquely recovered by writing the exponent of $b$ as $a +r$ with $0 \leq r < k$ and $k \mid a$. The $\lambda_i^{m_i}$ are then uniquely recovered by undoing the base $k$ expansion with the remaining terms (including $b^r$ when $r \not = 0$), i.e., extracting the largest power of $k$ dividing each term,  moving it to the exponent, and combining terms with the same base.
	
	Surjectivity is similar. Start with $\nu = (\nu_1^{n_1}, \nu_2^{n_2}, \ldots, \nu_\ell^{n_\ell}, b^\alpha)  \in C_{k,m}(n)$. Write $\alpha$ as $a + r$ with $0 \leq r < k$ and $k \mid a$. The preimage of $\nu$ is constructed via $a$, $b$, and the $\lambda_i^{m_i}$ resulting from undoing the base $k$ expansion on the remaining terms (including $b^r$ when $r \not = 0$) as above.
\end{proof}

\begin{remark}
Our bijection $ B_{2,m}(n) \longleftrightarrow C_{2,m}(n)$ provides an alternative proof of Theorem \ref{Franklin} for the case $j=1$. In particular,
note that this bijection extends canonically to a bijection $B_{2,0}(n) \dot\cup B_{2,1}(n) \longleftrightarrow C_{2,0}(n) \dot\cup C_{2,1}(n)$, where
$B_{2,0}(n) \dot\cup B_{2,1}(n)$ is the collection of all partitions $\lambda$ of $n$ with $\gamma_{\mathcal{O}}(\lambda) = 1$ and
$C_{2,0}(n) \dot\cup C_{2,1}(n)$ is the collection of all partitions $\lambda$ of $n$ with $\gamma_{\mathcal{D}}(\lambda) = 1$.
\end{remark}

\subsection{Bijection between $A_k(n) \, \dot\cup \, A'_k(n)$ and $C_{k,0}(n) \, \dot\cup \, C_{k,1}(n)$}

\begin{theorem}
	Let $$\mu =(\lambda = (\lambda_1^{m_1}, \lambda_2^{m_2}, \ldots, \lambda_\ell^{m_\ell}), \lambda_{i_0}, m) \in A_k(n) \, \dot\cup \, A'_k(n).$$ 	
	Write $\frown$ for sequence concatenation. The mapping
	$$ \mu \mapsto
	(\lambda_i^{m_i})_{i\ne i_0} \frown (\lambda_{i_0}^m, (m_{i_0}-m)^{\lambda_{i_0}}) $$
	gives a bijection between
	$$A_k(n) \, \dot\cup \, A'_k(n) \longleftrightarrow C_{k,0}(n) \, \dot\cup \, C_{k,1}(n).$$
	(Note that $(m_{i_0}-m)$ may also appear once among the $\lambda_i$, $i \not= i_0$. In that case, combine those two terms with the same base by adding the corresponding exponents.)
\end{theorem}

\begin{proof}
	For injectivity, first note that the $\lambda_{i_0}$ in the statement of the theorem is uniquely recovered from its image by writing the unique exponent that is at least $k$ as $\lambda_{i_0} +r$ where $0 \leq r < k$ and $k \mid \lambda_{i_0}$. Then $m$ is recovered uniquely as the exponent of $\lambda_{i_0}$ (so as $0$ if $\lambda_{i_0}$ does not appear). From that, $m_{i_0}$ is obtained from the base of the $\lambda_{i_0} +r$ exponent by adding $m$. After that, removing $\lambda_{i_0}^m$ and $(m_{i_0}-m)^{\lambda_{i_0}}$ from the partition determines the remaining $\lambda_i^{m_i}$, $i\ne i_0$.
	
	Surjectivity is similar. Start with $\nu = (\nu_1^{n_1}, \nu_2^{n_2}, \ldots, \nu_\ell^{n_\ell}, b^a)  \in C_{k,0}(n) \, \dot\cup \, C_{k,1}(n)$. To construct its preimage, write $a$ as $\lambda_{i_0} +r$ with $0 \leq r < k$ and $k \mid \lambda_{i_0}$. Let $m$ be the exponent of $\lambda_{i_0}$, and let $m_{i_0}$ be $b + m$. The rest of the $\lambda_i^{m_i}$, $i \not= i_0$, are chosen to be the remaining terms after removing $\lambda_{i_0}^m$ and $(m_{i_0}-m)^{\lambda_{i_0}}$ from the partition.	
\end{proof}

\subsection{Proof of Main Result}

\begin{proof}[\mbox{Combinatorial proof of Theorem \ref{main}}]
	This is immediate as we now have
	\begin{align*}
		b_{k,m}(n) = |B_{k,m}(n)| = |C_{k,m}(n)| = c_{k,m}(n),
	\end{align*}
while
	\begin{align*}
		a_k(n) + kc_{k,0}(n)
		= \, \mid A_k(n) \, \dot\cup \, A'_k(n) \mid
		= \, \mid C_{k,0}(n) \, \dot\cup \, C_{k,1}(n) \mid
		= c_{k,0}(n) + c_{k,1}(n),
	\end{align*}
gives
	\begin{align*}& a_k(n) = c_{k,1}(n) - (k-1)c_{k,0}(n). \qedhere \end{align*}
\end{proof}

\section{Connections to Andrews, Merca}

\begin{theorem} \label{AM}
	Comparing our notation to that of Andrews, Merca \cite{AM}, we have the following correlations.
	\begin{itemize}
		\item[$(a)$] $a(n) =a_2(n)$, $c_e(n) = c_{2,0}(n)$ and $c_o(n) = c_{2,1}(n)$.
		\item[$(b)$] $  b_e(n) =
		\begin{cases}
			b_{2,0}(n) & \text{if $n$ is even,}\\
			b_{2,1}(n) & \text{if $n$ is odd,}
		\end{cases} $
		and
		$  b_o(n) =
		\begin{cases}
			b_{2,1}(n) & \text{if $n$ is even,}\\
			b_{2,0}(n) & \text{if $n$ is odd.}
		\end{cases} $
	\end{itemize}
\end{theorem}

\begin{proof}
	$(a)$ follows trivially from the definitions of $a_k(n)$ and $c_{k,m}(n)$. For instance, by definition $a(n)$ and $a_2(n)$ both count the number of even parts in all partitions of $n$ into distinct parts. For $(b)$, let
$\lambda = (\lambda_1^{m_1}, \lambda_2^{m_2}, \ldots, \lambda_\ell^{m_\ell}, a^b)$ be an element of $B_{k,0}(n) \, \dot\cup \, B_{k,1}(n)$. If we take $n$ to be even, and if $2 \mid b$, then there must be an even number of odd parts in $\lambda$, and if $2 \nmid b$, then there must be an odd number of odd parts in $\lambda$. So, $b_{2,0}(n) = b_e(n)$ and $b_{2,1} = b_o(n)$. For $n$ odd, we argue similarly.
\end{proof}

From this, \cite[Theorem 1.4]{AM} and \cite[Theorem 1.5]{AM} become special cases of Theorem \ref{main} for $k=2$.

\section{The Generating Function for $a_k(n)$} \label{sec:gen fun}

For $k \ge 2$, the following calculation of formal power series provides the generating function of the sequence $a_k(n)$.
\begin{align*}
	& \sum_{n=0}^\infty a_k(n)q^n\\
	& = \frac{d}{dz} \bigg[
	\prod_{\substack{n=1\\ k \nmid n}}^{\infty} (1\!+\!q^n\!+\!(q^n)^2\!+\dots+\!(q^n)^{k-1} )
	\prod_{\substack{n=1\\ k \mid n}}^{\infty} (1\!+\!zq^n\!+\!(zq^n)^2\!+\dots+\!(zq^n)^{k-1})
	\bigg]_{\!z=1}\\
	& =
	\prod_{\substack{n=1\\ k \nmid n}}^{\infty} \frac{1-q^{kn}}{1-q^n}\cdot
	\frac{d}{dz} \bigg[
	\prod_{\substack{n=1\\ k \mid n}}^{\infty} \frac{1-z^kq^{kn}}{1-zq^n}
	\bigg]_{\!z=1}\\
	& =
	\prod_{\substack{n=1\\ k \nmid n}}^{\infty} \frac{1-q^{kn}}{1-q^n}\cdot
	\bigg[
	\prod_{\substack{n=1\\ k \mid n}}^{\infty} \frac{1-z^kq^{kn}}{1-zq^n} \cdot
	\sum_{\substack{n=1\\ k \mid n}}^{\infty} \frac{q^n\!+\!2q^n(zq^n)\!+\dots+\!(k-1)q^n(zq^n)^{k-2}}{1\!+\!zq^n\!+\!(zq^n)^2\!+\dots+\!(zq^n)^{k-1}}
	\bigg]_{\!z=1}\\
	& =
	\prod_{n=1}^{\infty} \frac{1-q^{kn}}{1-q^n}\cdot
	\sum_{n=1}^{\infty}
	\frac{q^{kn} + 2q^{2kn} +\dots+ (k-1)q^{(k-1)kn}}{1 + q^{kn} + q^{2kn} +\dots+ q^{(k-1)kn}}
\end{align*}

Table \ref{table:akn} lists $a_k(n)$ for some small values of $n$ and $k$. The sequence $a_2(n)$ is documented on OEIS \cite[A116680]{OEIS}.

\begin{table}[h]	
	\centering
	\scalebox{0.8}{
		\begin{tabular}{c|c|c|c|c|c|c|c|c|}
			\cline{2-8}
			& $k=2$ & $k=3$ & $k=4$ & $k=5$ & $k=6$ & $k=7$ & $k=8$ \\ \hline
			\multicolumn{1}{|l|}{$n=0$}  & $0$ & $0$ & $0$ & $0$  & $0$ & $0$ & $0$\\ \hline
			\multicolumn{1}{|l|}{$n=1$}  & $0$ & $0$ & $0$ & $0$  & $0$ & $0$ & $0$\\ \hline
			\multicolumn{1}{|l|}{$n=2$}  & $1$ & $0$ & $0$ & $0$ & $0$ & $0$ & $0$\\ \hline
			\multicolumn{1}{|l|}{$n=3$}  & $1$ & $1$ & $0$ & $0$ & $0$ & $0$ & $0$\\ \hline
			\multicolumn{1}{|l|}{$n=4$}  & $1$ & $1$ & $1$ & $0$ & $0$ & $0$ & $0$\\ \hline
			\multicolumn{1}{|l|}{$n=5$}  & $2$ & $2$ & $1$ & $1$ & $0$ & $0$ & $0$\\ \hline
			\multicolumn{1}{|l|}{$n=6$}  & $4$ & $4$ & $2$ & $1$ & $1$ & $0$ & $0$\\ \hline
			\multicolumn{1}{|l|}{$n=7$}  & $5$ & $6$ & $3$ & $2$ & $1$ & $1$ & $0$\\ \hline
			\multicolumn{1}{|l|}{$n=8$}  & $5$ & $9$ & $6$ & $3$ & $2$ & $1$ & $1$\\ \hline
			\multicolumn{1}{|l|}{$n=9$}  & $8$ & $10$ & $8$ & $5$ & $3$ & $2$ & $1$\\ \hline
			\multicolumn{1}{|l|}{$n=10$} & $11$ & $16$ & $13$ & $8$ & $5$ & $3$ & $2$\\ \hline
			\multicolumn{1}{|l|}{$n=11$} & $14$ & $21$ & $18$ & $12$ & $7$ & $5$ & $3$\\ \hline
			\multicolumn{1}{|l|}{$n=12$} & $18$ & $31$ & $26$ & $17$ & $12$ & $7$ & $5$\\ \hline
			\multicolumn{1}{|l|}{$n=13$} & $23$ & $39$ & $36$ & $25$ & $16$ & $11$ & $7$\\ \hline
			\multicolumn{1}{|l|}{$n=14$} & $29$ & $54$ & $51$ & $35$ & $24$ & $16$ & $11$\\ \hline
			\multicolumn{1}{|l|}{$n=15$} & $37$ & $69$ & $68$ & $48$ & $33$ & $23$ & $15$\\ \hline
		\end{tabular}
	}	
	\caption{Small Values of $a_k(n)$}
	\label{table:akn}
\end{table}

\begin{remark}
	It can easily be verified from the definitions that
	$$
	a_k(n) =
	\begin{cases}
		 0 & \text{for $1\le n <k$},\\
		 p(n-k) & \mbox{for $k\le n < 2k$},\\
		 p(k) + 1 & \mbox{for $n = 2k$}
	\end{cases}
	$$
	for all $k\ge 3$, where $p(n)$ denotes Euler's partition function.
	These patterns are manifest in Table \ref{table:akn}.
\end{remark}

\section{Recurrence Relations for $a_k(n)$} \label{sec: recur reln}

One deduces the following two recurrence relations for $a_k(n)$ from the generating function, from which setting $k=2$ recovers those found in \cite{AM}. Note that these recurrence relations are also satisfied by $p(n)$.

\begin{theorem}\label{t5.2}
For $k \nmid n$, the sequence $a_k(n)$ satisfies the generalized pentagonal number recurrence relation
\begin{align*}
a_k(n)&=a_k(n-1)+a_k(n-2)-a_k(n-5)-a_k(n-7)\\
&\quad\, +a_k(n-12)+a_k(n-15)-a_k(n-22)-a_k(n-26)+\cdots.
\end{align*}
\end{theorem}

\begin{proof}
From the generating function for $a_k(n)$, one can write
\begin{align*}
(q;q)_\infty\sum_{n=0}^\infty a_k(n)q^n=(q^k;q^k)_\infty\sum_{n=1}^{\infty}
	\frac{q^{kn} + 2q^{2kn} +\dots+ (k-1)q^{(k-1)kn}}{1 + q^{kn} + q^{2kn} +\dots+ q^{(k-1)kn}},
\end{align*}
where $(a;q)_\infty$ represents the \emph{$q$-series} or \emph{$q$-Pochhammer symbol}
\begin{align*}
(a;q)_\infty=\prod_{n=0}^\infty(1-aq^n).
\end{align*}
Substituting Euler's pentagonal number theorem \cite[Corollary 1.7]{partitions} on the l.h.s.,
\begin{align*}
(q;q)_\infty=\sum_{n=-\infty}^{\infty} (-1)^nq^{n(3n-1)/2},
\end{align*}
and noting that the r.h.s. is a power series in $q^k$ leads to the recurrence relation. Namely, the coefficients of all $q^n,\ k \nmid n,$ are equal to zero, hence it quickly follows that
\begin{align*}
0&=a_k(n)-a_k(n-1)-a_k(n-2)+a_k(n-5)+a_k(n-7)-\cdots.\qedhere
\end{align*}
\end{proof}

\begin{theorem}\label{t5.3}
For $k$ even, $n$ odd, the sequence $a_k(n)$ satisfies the triangular number recurrence relation
\begin{align*}
a_k(n)&=a_k(n-1)+a_k(n-3)-a_k(n-6)-a_k(n-10)\\
&\quad\, +a_k(n-15)+a_k(n-21)-a_k(n-28)-a_k(n-36)+\cdots.
\end{align*}
\end{theorem}

\begin{proof}
From the generating function for $a_k(n)$, one can further write
\begin{align*}
(q;q)_\infty&(-q^2;q^2)_\infty  \sum_{n=0}^\infty a_k(n)q^n\\
&=(-q^{2};q^{2})_\infty(q^{k};q^{k})_\infty\sum_{n=1}^{\infty}
	\frac{q^{kn} + 2q^{2kn} +\dots+ (k-1)q^{(k-1)kn}}{1 + q^{kn} + q^{2kn} +\dots+ q^{(k-1)kn}}.
\end{align*}
Substituting on the l.h.s. the theta identity \cite[Equation (2.2.13)]{partitions},
\begin{align*}
(q;q)_\infty(-q^2;q^2)_\infty=\dfrac{(q^2;q^2)_\infty}{(-q;q^2)_\infty}=\sum_{n=0}^\infty (-q)^{n(n+1)/2},
\end{align*}
and noting that the r.h.s. is a power series in $q^2$ leads to the recurrence relation. Namely, the coefficients of all $q^n,$ $k$ even, $n$ odd, are equal to zero, hence it quickly follows that
\begin{align*}
0&=a_k(n)-a_k(n-1)-a_k(n-3)+a_k(n-6)+a_k(n-10)-\cdots.\qedhere
\end{align*}
\end{proof}

We include another particularly nice set of recurrence relations for $a_k(n)$ which involves square powers and their doubles. Once again, these recurrence relations are shared by $p(n)$ (see \cite[Theorem 1]{CKS}).

\begin{theorem}\label{t6.3}
For $k$ even, $n$ odd, the sequence $a_k(n)$ satisfies the square power recurrence relation
\begin{align*}
a_k(n)=\sum _{m=1}^{\left\lfloor \sqrt{n}\right\rfloor } (-1)^{m-1}a_k(n-m^2)+\sum _{m=1}^{\left\lfloor \sqrt{\frac{n}{2}}\right\rfloor } (-1)^{m-1}a_k(n-2m^2).
\end{align*}
Thus
\begin{align*}
a_k(n)&=a_k(n-1)+a_k(n-2)-a_k(n-4)-a_k(n-8)\\
&\quad\, +a_k(n-9)+a_k(n-18)-a_k(n-16)-a_k(n-32)+\cdots.
\end{align*}
\end{theorem}

\begin{proof}
From the generating function for $a_k(n)$ we define the series
\begin{align} \label{eq1}
&f_k(q) = (q;q)_\infty (q;q^2)_\infty  \sum_{n=0}^\infty a_k(n)q^n\\ \notag
&\qquad\ =(q;q^2)_\infty (q^{k};q^{k})_\infty\sum_{n=1}^{\infty}
	\frac{q^{kn} + 2q^{2kn} +\dots+ (k-1)q^{(k-1)kn}}{1 + q^{kn} + q^{2kn} +\dots+ q^{(k-1)kn}}.
\end{align}
Using Euler's identity
\begin{align*}
(-q;q)_\infty=\dfrac{1}{(q;q^2)_\infty},
\end{align*}
one substitutes Gauss's square power identity \cite[Equation (2.2.12)]{partitions} on the l.h.s,
\begin{align*}
(q;q)_\infty(q;q^2)_\infty =\dfrac{(q;q)_\infty}{(-q;q)_\infty}=\sum_{m=-\infty}^\infty (-1)^m q^{m^2}.
\end{align*}
Thus, the coefficient of $q^n$ of the series $f_k(q)$ is given by
\begin{align} \label{eq2}
\sum_{m=-\infty}^\infty (-1)^m a_k(n-m^2) = a(n) - 2\sum_{m= 1}^{\left\lfloor \sqrt{n}\right\rfloor} (-1)^{m-1} a_k(n-m^2).
\end{align}

Next we consider the series
\begin{align} \label{eq3}
&g_k(q) = (q;q)_\infty (-q;q^2)_\infty  \sum_{n=0}^\infty a_k(n)q^n\\ \notag
&\qquad\ =(-q;q^2)_\infty (q^{k};q^{k})_\infty\sum_{n=1}^{\infty}
	\frac{q^{kn} + 2q^{2kn} +\dots+ (k-1)q^{(k-1)kn}}{1 + q^{kn} + q^{2kn} +\dots+ q^{(k-1)kn}}.
\end{align}
Using once more Gauss's square power identity, on the l.h.s we can substitute
\begin{align*}
(q;q)_\infty(-q;q^2)_\infty =\dfrac{(q^2;q^2)_\infty}{(-q^2;q^2)_\infty}=\sum_{m=-\infty}^\infty (-1)^m (q^2)^{m^2}=\sum_{m=-\infty}^\infty (-1)^m q^{2m^2},
\end{align*}
and the coefficient of $q^n$ of the series $g_k(q)$ is given by
\begin{align} \label{eq4}
\sum_{m=-\infty}^\infty (-1)^m a_k(n-2m^2) = a(n) - 2\sum_{m= 1}^{\left\lfloor \sqrt{\frac{n}{2}}\right\rfloor} (-1)^{m-1} a_k(n-2m^2).
\end{align}
Comparing the r.h.s of Equations \eqref{eq1} and \eqref{eq3}, we see $g_k(q) = f_k(-q)$. Hence, as $n$ is odd,
the expressions in Equations \eqref{eq2} and \eqref{eq4} differ only by a sign. In particular, adding gives
\begin{align*}
2a(n) - 2\sum_{m= 1}^{\left\lfloor \sqrt{n}\right\rfloor} (-1)^{m-1} a_k(n-m^2) - 2\sum_{m= 1}^{\left\lfloor \sqrt{\frac{n}{2}}\right\rfloor} (-1)^{m-1} a_k(n-2m^2) = 0,
\end{align*}
and the desired recurrence follows.
\end{proof}

Theorem \ref{t6.3} has the following surprising counterpart for the case $k$ and $n$ even.

\begin{theorem}\label{t6.4}
For $k$ and $n$ both even, we have the identity
\begin{align*}
\sum _{m=1}^{\left\lfloor \sqrt{n}\right\rfloor } (-1)^{m-1}a_k(n-m^2) = \sum _{m=1}^{\left\lfloor \sqrt{\frac{n}{2}}\right\rfloor } (-1)^{m-1}a_k(n-2m^2).
\end{align*}
\end{theorem}

\begin{proof}
We define the series $f_k(q)$ and $g_k(q)$ with $g_k(q) = f_k(-q)$ as in the proof of Theorem \ref{t6.3}. This time $n$ is even, and
the expressions in Equations \eqref{eq2} and \eqref{eq4} must be equal. In particular,
\begin{align*}
a(n) - 2\sum_{m= 1}^{\left\lfloor \sqrt{n}\right\rfloor} (-1)^{m-1} a_k(n-m^2) = a(n) - 2\sum_{m= 1}^{\left\lfloor \sqrt{\frac{n}{2}}\right\rfloor} (-1)^{m-1} a_k(n-2m^2),
\end{align*}
and the desired identity follows.
\end{proof}

As would be expected, the identity of Theorem \ref{t6.4} is once again shared by $p(n)$.

\begin{theorem}\label{t6.5}
For $n$ even, we have the identity
\begin{align*}
\sum _{m=1}^{\left\lfloor \sqrt{n}\right\rfloor } (-1)^{m-1}p(n-m^2) = \sum _{m=1}^{\left\lfloor \sqrt{\frac{n}{2}}\right\rfloor } (-1)^{m-1}p(n-2m^2) = \frac{p(n) - p_{\mathcal{DO}}(n)}2,
\end{align*}
where $p_{\mathcal{DO}}(n)$ denotes the number of partitions of $n$ into distinct odd parts.
\end{theorem}

\begin{proof}
The proof of the first equality is similar to Theorems \ref{t6.3} and \ref{t6.4}, using the series
\begin{align*}
&(q;q)_\infty (q;q^2)_\infty  \sum_{n=0}^\infty p(n)q^n\! =(q;q^2)_\infty \mbox{ and } (q;q)_\infty (-q;q^2)_\infty  \sum_{n=0}^\infty p(n)q^n\! =(-q;q^2)_\infty.
\end{align*}
The second equality then immediately follows from an identity in \cite[Theorem 1]{CKS} for $n$ even,
\begin{equation*}
p(n) - \sum _{m=1}^{\left\lfloor \sqrt{n}\right\rfloor } (-1)^{m-1}p(n-m^2) - \sum _{m=1}^{\left\lfloor \sqrt{\frac{n}{2}}\right\rfloor } (-1)^{m-1}p(n-2m^2) = p_{\mathcal{DO}}(n).\hfill \mbox{} \qedhere
\end{equation*}
\end{proof}

\section{Concluding Remarks} \label{concluding}

When considering the proof of Theorem \ref{t5.3}, one notices that multiplying by $(-q^2;q^2)_\infty$ leads directly to the triangular number recurrence relation. This naturally raises the question of what recurrence relations occur, if any, when one multiplies instead by $(-q^m;q^m)_\infty,\ 3\leq m\leq k.$ In particular, for $m\mid k$, one obtains
\begin{align*}
(q;q)_\infty&(-q^m;q^m)_\infty  \sum_{n=0}^\infty a_k(n)q^n\\
&=(-q^{m};q^{m})_\infty(q^{k};q^{k})_\infty\sum_{n=1}^{\infty}
	\frac{q^{kn} + 2q^{2kn} +\dots+ (k-1)q^{(k-1)kn}}{1 + q^{kn} + q^{2kn} +\dots+ q^{(k-1)kn}},
\end{align*}
where the r.h.s. is now a power series in $q^m$, which implies resulting recurrence relations for all $m \nmid n$. However, writing a closed expression for these relations would require a closed power series representation for $(q;q)_\infty(-q^m;q^m)_\infty,\ m\geq3$. While using \emph{Mathematica} \cite{math} allows one to find these relations for specific values of $m,k,n,$ to the best of our knowledge there does not appear to exist any closed form for the recurrence relations for $m\geq3$.

Similarly, Theorems \ref{t6.3} and \ref{t6.4} originally resulted from multiplying the generating function for $a_k(n)$ by either $( q;q^2)_\infty$ or $(-q;q^2)_\infty$ and investigating the resulting recurrence relations
\begin{align*}
a_k(n)=2\sum _{m=1}^{\left\lfloor \sqrt{n}\right\rfloor } (-1)^{m-1}a_k(n-m^2),\quad a_k(n)=2\sum _{m=1}^{\left\lfloor \sqrt{\frac{n}{2}}\right\rfloor } (-1)^{m-1}a_k(n-2m^2).
\end{align*}
Neither of these is generally true, but it came as quite a surprise that for even $k$ and arbitrary $n$ the one recurrence relation holds if and only if the other holds, which is now an easy consequence of Theorems \ref{t6.3} and \ref{t6.4}.
It is also interesting to note that both recurrence relations seem to hold quite often when considering $k=2,4$. In particular, for $k=2$, both recurrences are valid for the following list of values $n\le 500$:
\begin{eqnarray*}
& 1,5,11,22,23,27,30,41,61,65,66,71,72,79,93,100,115,116,117,120,\\
& 122,124, 131,135,166,183,191,203,204,214,216,223,224,229,236,\\
& 252,258,262,286,288,289,291,311,316,321,324,331,336,341,350,361,\\
& 366,367,378,383,390,403,414,416,418,425,440,468,470,471,488.
\end{eqnarray*}
Note that these are precisely the values of $n$ for which the term $q^n$ is missing in the power series
\begin{align*}
(q;q^2)_\infty (q^{2};q^{2})_\infty\sum_{n=1}^{\infty} \frac{q^{2n}}{1 + q^{2n}}= (q;q)_\infty \sum_{n=1}^{\infty} \frac{q^{2n}}{1 + q^{2n}}.
\end{align*}
For $k=4$, both recurrences are valid for the following list of values $n\le 500$:
\begin{align*}
&1, 2, 3, 6, 19, 53, 54, 58, 99, 143, 164, 170, 173, 257, 283, 302, 328, 338, 356, 359, 376, 473.
\end{align*}
We do not currently observe a pattern appearing in either of these lists.
This leads to a few interesting questions:
\begin{itemize}
\item[(1)] Are the sets of valid $n$ values for $k =2,4$ finite or infinite?
\item[(2)] Is there a way to characterize these sets in a closed form?
\end{itemize}

We conclude by remarking that it may be interesting to investigate analogies for $a_k(n)$ of the other relations for $p(n)$ in \cite{CKS} and that the conjectured inequalities regarding $a(n)=a_2(n)$ in Andrews, Merca \cite[Conjecture 1.6]{AM} extend quite naturally to $a_k(n),\ k\geq3$. Combinatorial proofs for our results in \S \ref{sec: recur reln} are very welcome.


\end{document}